\documentclass[a4paper,12pt]{article}
  \usepackage[latin1]{inputenc}
  \usepackage[T1]{fontenc}
  \usepackage{amsmath,amsthm,amssymb,epsfig,graphicx}
  \usepackage[all]{xy}

\input epsf

\def\R{\mathbb{R}}
\def\C{\mathbb C}
\def\Z{\mathbb Z}
\def\N{\mathbb N}

\def\Q{ \mathbb{Q}}
\def\Kappa{\mathrm{K}}
\usepackage{mathrsfs}
\def\M{{\mathscr M}}
\newcommand{\A}{\mathcal{A}}
\def\const{\operatorname{const}}

\renewcommand{\epsilon}{\varepsilon}
\title{On the finite cyclicity of open period annuli}
 \author{Lubomir Gavrilov \\
 \normalsize \it Institut de Math\'{e}matiques de Toulouse, UMR 5219\\
 \normalsize \it Universit\'{e}  de Toulouse \\  \normalsize \it 31062 Toulouse, France\\
 \\ Dmitry Novikov\\
  \normalsize \it Department of Mathematics\\
  \normalsize \it Weizmann Institute of Science\\
  \normalsize \it Rehovot, ISRAEL
  }
\date{July 2, 2008}
\begin{document}
\maketitle
\newtheorem{definition}{Definition}
\newtheorem{remark}{Remark}
\newtheorem{example}{Example}
\newtheorem{theorem}{Theorem} 
\newtheorem{lemma}{Lemma}
\newtheorem{proposition}{Proposition}
\newtheorem{corollary}{Corollary}
\vspace{5mm}
\begin{abstract}
Let $\Pi$ be an open, relatively compact period annulus of real analytic vector field $X_0$ on an analytic
surface. We prove that the maximal number of limit cycles which bifurcate from $\Pi$ under a given
multi-parameter analytic deformation $X_\lambda$ of $X_0$ is finite, provided that $X_0$ is either Hamiltonian,
or generic Darbouxian vector field.
\end{abstract}

\section{Statement of the result}
Let $S$ be a real analytic surface without border (compact or not), and $X_0$ a
real analytic vector field on it. An open period annulus of $X_0$ is an union
of period orbits of $X_0$ which is bi-analytic to the standard open annulus
$S^1\times (0,1)$, the image of each circle $S^1\times \{u\}$ being a periodic
orbit of $X_0$.

Let $X_\lambda$, $\lambda\in (\R^n,0$) be an analytic family of analytic vector fields on  $S$, and let
 $\Pi$ be an open period annulus of $X_0$.  The \emph{cyclicity} $Cycl(\Pi,X_\lambda)$ of  $\Pi$ with respect to
the deformation $X_\lambda$ is the maximal number of limit cycles of $X_\lambda$ which tend to $\Pi$ as
$\lambda$ tends to zero, see Definition \ref{cyclicity} bellow. Clearly the vector field $X_0$ has an analytic
first integral $f$ in the period annulus $\Pi$ which has no critical points. In what follows we shall suppose
that
 the open period annulus $\Pi$ is relatively compact (i.e. its closure $\bar{\Pi}\subset S$ is
    compact).
\begin{definition}
We shall say that $X_0$ is a Hamiltonian vector field provided that it has a
first integral with isolated critical points in a complex neighborhood of
$\Pi$. We shall say that $X_0$ is a generic Darbouxian vector field provided
that all singular points of $X_0$ in a neighborhood of $\bar{\Pi}$ are
orbitally analytically equivalent to linear saddles $\dot x=\lambda x, \dot
y=-y$ with $\lambda>0$.
\end{definition}
\begin{remark}
In the case when $X_0$ is a generic Darbouxian vector field, as we shall
see in the next section, it can be covered by a planar Darbouxian vector field with a first integral of the "Darboux type"
$H=\prod_{i=1}^n P_i^{\lambda_i}$ for some analytic functions $P_i$ in a
complex neighborhood of $\Pi$.
\end{remark}

The main result of the paper are the following
\begin{theorem}
\label{main1} The cyclicity $Cycl(\Pi,X_\lambda)$ of the open period annulus
$\Pi$ of a Hamiltonian  vector field $X_0$ is finite.
\end{theorem}
\begin{theorem}
\label{main2} The cyclicity $Cycl(\Pi,X_\lambda)$ of the open period annulus
$\Pi$ of a  generic Darbouxian vector field $X_0$ is finite.
\end{theorem}
The above theorems are a particular case of the Roussarie's conjecture
\cite[p.23]{rous98} which claims that the cyclicity $Cycl(\Gamma,X_\lambda)$ of
every compact invariant set $\Gamma$ of $X_0$ is finite. Indeed, as $\Pi$ is
relatively compact, then $Cycl(\Pi,X_\lambda) \leq Cycl(\bar{\Pi},X_\lambda)$.
The finite cyclicity of the open period annulus without the assumptions of
Theorems \ref{main1} and \ref{main2} is an open question.

To prove the finite cyclicity we note first that it suffices to show the finite
cyclicity of a given one-parameter deformation $X_\varepsilon$. This argument
is based on the Hironaka's desingularization theorem, see \cite{rous01,gav0}.
Consider the first return map associated to $\Pi$ and $X_\varepsilon$
$$
t\rightarrow t+ \varepsilon^k M_k(t)+\dots, t\in (0,1), \varepsilon \sim 0 .
$$
The cyclicity of the open period annulus $\Pi$ is finite if and only if the
Poincar\'{e}-Pontryagin function $M_k$ has a finite number of zeros in $(0,1)$. It
has been shown in \cite{gav} that $M_k$ allows an integral representation as a
linear combination of iterated path integrals along the ovals of $\Pi$ of
length at most $k$. The finite cyclicity  follows then from the
non-accumulation of zeros of such iterated integrals at $0$ and $1$. The proof
of this fact will be different in the Hamiltonian and in the generic Darbouxian
case.

In the Hamiltonian case we observe that $M_k$ satisfies a Fuchsian equation
\cite{gav1,gav}. We prove in section \ref{proof1} that the associated monodromy
representation is quasi-unipotent, which implies the desired property.

In the Darbouxian case the above argument does not apply (there is no Fuchsian
equation satisfied by $M_k$). We prove the non-oscillation property of an
iterated integral by making use of its Mellin transformation, along the lines
of \cite{novikov}. It seems to be difficult to remove the genericity assumption
in the Darbouxian case (without this the Hamiltonian case is a sub-case of the
Darbouxian one).

The paper is organized as follows. In section \ref{ppm} we recall the
definition of cyclicity and the reduction of multi-parameter to one-parameter
deformations. In section \ref{plane} we reduce the case of a vector field on a
surface to the case of a plane vector field. Theorem \ref{main1}  is proved in
section \ref{proof1} according to the scheme
 $$\mbox{Proposition \ref{quasi}} \Rightarrow   \mbox{Proposition \ref{prop3}}  \Rightarrow \mbox{Proposition \ref{th4} }  $$
$$\{ \mbox{Theorem \ref{th2}  + Proposition \ref{th4}} \} \Rightarrow \mbox{Theorem \ref{th3}}$$
$$
\mbox{Theorem \ref{th3}} \Rightarrow \mbox{Theorem \ref{main1}} .$$
Theorem~\ref{main2}  is proved in section~\ref{proof2} of the paper.

{\bf Acknowledgments.} Part of this paper was written while the second author
was visiting  the University of Toulouse (France). He is obliged for the
hospitality.

\section{Cyclicity and non-oscillation of the Poincar\'{e}-Pontryagin-Melnikov function}
\label{ppm}
\begin{definition}
\label{cyclicity}
 Let $X_\lambda$ be a family of analytic real  vector fields on a surface $S$, depending analytically on a
parameter $\lambda \in (\R^n,0)$, and let $ K \subset S$ be a compact invariant set of $X_{\lambda_0}$. We say
that the pair $(K, X_{\lambda_0}$) has cyclicity $N = Cycl((K,X_{\lambda_0}), X_\lambda)$ with respect to the
deformation $X_\lambda$, provided that $N$ is the smallest integer having the property: there exists
$\varepsilon_0 > 0$ and a neighborhood $V_K$ of $K$, such that for every $\lambda$, such that $\|\lambda-
\lambda_0\| < \varepsilon_0$, the vector field $X_\lambda$  has no more than N limit cycles contained in $V_K$.
If $\tilde{K}$ is an invariant set of $X_{\lambda_0}$ (possibly non-compact), then the cyclicity of the pair
$(\tilde{K} , X_{\lambda_0})$ with respect to the deformation $X_\lambda$ is
$$
Cycl((\tilde{K} , X_{\lambda_0}),X_\lambda) = sup\{Cycl((K,X_{\lambda_0}), X_\lambda) : K \subset  \tilde{K} , K
\mbox{   is a compact }\}.
$$
\end{definition}
The above definition implies that when $\tilde{K} $ is an open invariant set, then its cyclicity $Cycl((\tilde{K} , X_{\lambda_0}), X_\lambda)$ is  the maximal number of limit cycles
which tend to $\tilde{K}$ as $\lambda$ tends to $0$.
To simplify the notation, and if there is no danger of
confusion, we shall write $ Cycl(K  ,X_\lambda) $ on the place of $ Cycl((K , X_{\lambda_0}),X_\lambda) . $

\begin{example}
Let $f_\varepsilon(t)= \varepsilon e^{-1/t}(t sin(1/t) - \varepsilon), f_\varepsilon(0)=0$. One can easily see that $f_\varepsilon(t)=0$ has finite number of isolated positive zeros for each $\varepsilon$, and this number tends to infinity as $\varepsilon\to 0$. Below we construct a germ $X_\varepsilon$ of a vector field having a monodromic planar singular point at the origin, with a return map $x\to x+f_\varepsilon(x)$. Since isolated singular points of the return map correspond  to limit cycles, we see that the vector field $X_\varepsilon$ has  a finite number of limit cycles for each $\varepsilon$, and this number tends to infinity as $\varepsilon$ tends to zero. So the cyclicity of the open period annulus $\Pi = \R^2\setminus \{0\}$ is infinity. Note that, however, the vector field $X_\varepsilon$ is not analytic at the origin.

Here is a construction: on the strip $S=[0,\delta]\times \R$ consider the equivalence relation $(r,\phi)\sim(r+f_\varepsilon(r), \phi-2\pi)$. Let $p:S\to S/\sim$ be the corresponding projection,  and define $\tilde{X}_\varepsilon=p_*(\partial_\phi)$. One can check that for $\delta$ small enough thus defined $\tilde{X}_\varepsilon$ is a blow-up of a smooth vector field $X_\varepsilon$ defined near the origin, and the  return map of $X_\varepsilon$ is as prescribed by construction.
\end{example}

Let $\Delta \subset S$ be a  cross-section of the period annulus $\Pi$ which can be identified to the interval
$(0,1)$. Choose a local parameter  $u$ on $\Delta$. Let $u \mapsto P(u,\lambda)$ be the first return map
 and $\delta(u,\lambda)= P(u,\lambda) - u$ the displacement function of $X_\lambda$. For every closed interval
 $[a,b]\subset \Delta$ there exists $\varepsilon_0 > 0$ such that the displacement function
 $\delta(u,\lambda)$ is well defined and analytic in
 $\{ (u,\lambda):    a-\varepsilon_0 < u < b+ \varepsilon_0, \| \lambda
 \|<\varepsilon_0\}$. For every fixed $\lambda$ there is a one-to-one correspondance between isolated zeros of
 $\delta(u,\lambda)$ and limit cycles of
 the vector field $X_\lambda$.

 Let $u_0\in
\Delta$ and let us expand
$$
\delta(u,\lambda)= \sum_{i=0}^\infty a_i(\lambda) (u-u_0)^i.
$$
\begin{definition}[Bautin ideal \cite{rous89}, \cite{rous98}]
We define the Bautin ideal $\mathcal{I}$ of $X_\lambda$ to be the ideal generated by the germs $\tilde{a}_i$ of
$a_i$ in the local ring $\mathcal{O}_0(\R^n)$ of analytic germs of functions at $0\in \R^n$.
\end{definition}
This ideal is Noetherian. Let $\tilde{\varphi}_1,\tilde{\varphi}_2,\dots,\tilde{\varphi}_p$ be a minimal
system of its generators, where $p = \dim_\R \mathcal{I} / \mathcal{MI}$, and $\mathcal{M}$ is the maximal ideal of
the local ring $\mathcal{O}_0(\R^n)$. Let $\varphi_1,\varphi_2,\dots,\varphi_p$ be analytic functions
representing the generators  of the Bautin ideal in a neighborhood of the origin in $\R^n$.
\begin{proposition}[Roussarie, \cite{rous98}]
\label{bautin}
 The Bautin ideal
does not depend on the point $u_0\in \Delta$. For every $[a,b] \subset \Delta$ there is an open neighborhood $U$
of $[a,b]\times \{0\}$ in $\R\times \R^n$ and analytic functions $h_i(u,\lambda)$ in $U$, such that
\begin{equation}\label{b1}
\delta(u,\lambda)= \sum_{i=0}^p \varphi_i(\lambda) h_i(u,\lambda) .
\end{equation}
The real vector space generated by the functions $h_i(u,0), u\in [a,b]$ is of dimension $p$.
\end{proposition}
Suppose that the Bautin ideal is principal and generated by $\varphi(\lambda)$. Then
\begin{equation}\label{principal}
    \delta(u,\lambda)= \varphi(\lambda) h(u,\lambda)
\end{equation}
where $h(u,0)\not\equiv0$. The maximal number of the isolated zeros of $h(u,\lambda)$ on a closed interval
$[a,b]\subset (0,1)$  for sufficiently small $|\lambda|$ is bounded by the number of the zeros of $h(u,0)$,
counted with multiplicity, on $[a,b]$. This follows from the Weierstrass preparation theorem, properly applied,
see \cite{gav0}. Therefore to prove the finite cyclicity of $\Pi$ it is enough to show that $h(u,0)$ has a
finite number of zeros on $(0,1)$. Consider a germ of analytic curve $\xi : \varepsilon \mapsto
\lambda(\varepsilon)$, $\lambda(0)=0$, as well the analytic one-parameter family of vector fields
$X_{\lambda(\varepsilon)}$. The Bautin ideal is principal, $\delta(u,\varepsilon)= \varphi(\varepsilon)
h(u,\varepsilon)$, and
$$
 \delta(u,\lambda(\varepsilon))=  \varepsilon^k  M_k(u) + \dots,  M_k(u)= c \, h(u,0), c\neq 0
$$
where the dots stay for terms containing $\varepsilon^i$, $i\geq k$. $M_k$ is the so called $k$th order higher
Poincar\'{e}-Pontryagin-Melnikov function associated to the one-parameter deformation $X_{\lambda(\varepsilon)}$ of
the vector field $X_0$. Therefore, if the cyclicity of the open period annulus is  infinite, then $M_k $ has an
infinite number of zeros on the interval $(0,1)$

Of course, in general  the Bautin ideal is not principal. However, by making use of the Hironaka's theorem, we can
always principalize it. More precisely, after several blow up's of the origin of the parameter space, we can
replace the Bautin ideal by an ideal sheaf which is principal, see \cite{gav0} for details. This proves the
following
\begin{proposition}
\label{ppm} If the cyclicity $ Cycl(K  ,X_\lambda) $ of the open period annulus $\Pi$ is infinite, then there
exists a one parameter deformation $\lambda=\lambda(\varepsilon)$, such that the corresponding higher order
Poincar\'{e}-Pontryagin-Melnikov function $M_k $ has an infinite number of zeros on the interval $(0,1)$.
\end{proposition}

In the next two sections we shall prove the non-oscillation property of $M_k$ in the Hamiltonian and the
Darbouxian case (under the restrictions stated in Theorem \ref{main2}).

\section{Reduction to the case of a plane vector field}
\label{plane}

Let $X_0$ be a real analytic vector field on a real analytic surface $S$. Let
$\Pi$ be an open period annulus of $X_0$ with compact closure. Let the map
$\tau:\Pi\to S^1\times (0,1)$ be a bi-analytic isomorphism, such that
$\delta_t=\tau^{-1}\left(S^1\times\{t\}\right)$ is a closed orbit of  $X_0$. We assume that $X_0$ is either Hamiltonian or generalized Darbouxian in some neighborhood of the closure
$\bar{\Pi}$ of $\Pi$. Theorems  \ref{main1} and \ref{main2} claims that cyclicity of
$\Pi$ in any family of analytic deformation $X_{\lambda}$ of $X_0$ is finite.

This paragraph is devoted to the reduction of this general situation to the
case of a vector field $X_0$ on $\R^2$ of Hamiltonian or Darboux type near its
polycycle. Then Theorem~\ref{main1} and Theorem~\ref{main2} follow from Theorem~\ref{main1'} and Theorem~\ref{main2'}
below.

First, note that it is enough to prove finite cyclicity of
$\tau^{-1}\left(S^1\times(0,\epsilon)\right)$ only. Indeed, finite cyclicity of
$\tau^{-1}\left(S^1\times[\epsilon, 1-\epsilon]\right)$ follows from
Gabrielov's theorem, and finite cyclicity of
$\tau^{-1}\left(S^1\times(1-\epsilon,1)\right)$ can be reduced to the above by
replacing $t$ by $1-t$.

Consider the Hausdorf limit
$\Gamma=\lim_{t\to0}\tau^{-1}\left(S^1\times\{t\}\right)$. It is a connected
union of several fixed points $a_1,...,a_n$ of $X_0$ (not necessarily pairwise
different) and orbits $\Gamma_1,..., \Gamma_n$ of $X_0$ such that $\Gamma_i$
exits from $a_i$ and enters $a_{i+1}$ (where $a_{n+1}$ denotes $a_1$).

From now on we consider only a sufficiently small neighborhood $U$ of $\Gamma$.
We assume that $U\cap\Pi=\tau^{-1}\left(S^1\times(0,\epsilon)\right)$, and
denote this intersection again by $\Pi$. We consider first the Darbouxian case. Note that $\Gamma$ cannot consist of just one singular point of $X_0$ by assumption about linearizability of singular points of $X_0$ in this case.

\begin{lemma}\label{lem:covering}
Assume that Theorem~\ref{main2} holds if $U$ is orientable and all $a_i$ are
different. Then Theorem~\ref{main2} holds in full generality.
\end{lemma}

\begin{proof}
Assume that for some real analytic surface $\tilde{U}$ there is an analytic
mapping $\pi:\tilde{U}\to U$ which is a finite covering on $\Pi$. Then the
cyclicity of $\Pi$ for $X_\lambda$ is the same as cyclicity of $\pi^{-1}(\Pi)$
for the  lifting $X_\lambda$ to $\tilde{U}$. The claim of the Lemma follows
from this principle applied to two types of coverings below.

First, taking a double covering of $U$ as $\tilde{U}$, we can assume that $U$
is orientable.

Second, let $U$ be represented as a union of neighborhoods $U_i$ of $a_i$
together with neighborhoods $V_i$ of $\Gamma_i$. Glue $\tilde{U}$ as
$\tilde{U}=\tilde{U_1}\cup \tilde{V_1}\cup ...\cup \tilde{V_n}$, where
$\tilde{U_i}$ are bianalytically equivalent to $U_i$ and disjoint, and
$\tilde{V_i}$ are bianalytically equivalent to $V_i$, with natural glueing of
$\tilde{U_i}$ to $\tilde{V_i}$,   of $\tilde{V_i}$ to $\tilde{U_{i+1}}$ and of
of $\tilde{U_1}$ to $\tilde{V_n}$. In other words, $\pi:\tilde{U}\to U$ is
one-to-one away from $a_i$ and $k_i$-to-one in a neighborhood of $a_i$ if $a_i$
appears $k_i$ times in the list $\{a_1,....,a_n\}$. Evidently, $\pi$ is
one-to-one on $\Pi$, so is bianalytic.
\end{proof}

\begin{figure}
\centerline{\epsfysize=0.25\vsize\epsffile{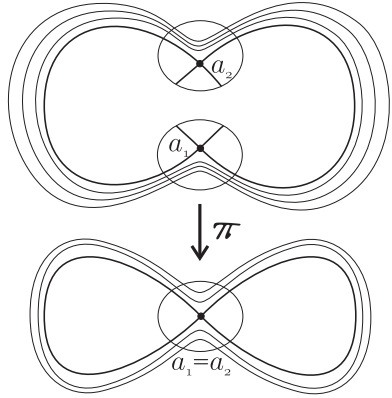}} \caption{Proof of
Lemma~\ref{lem:covering}.}\label{fig:covering}
\end{figure}

We will now define a first integral $H$ of $X_0$ in $U$. Take any non-singular
point $a\in\gamma_1$, and let $H$ be a local first integral of $X_0$ in a
neighborhood $U_a$ of $a$  such that $H(a)=0$ and $dH(a)\not =0$. Since $U$ is
orientable, $\Pi$ lies from one side of $\Gamma$, and we can assume that
intersection of $U_a$ with  each cycle $\delta_t$ is connected. This allows to
extend $H$ to a first integral of $X_0$ defined on $\Pi\cap U$. Changing sign
of $H$ if necessary, we can assume that $H>0$ on $\Pi\cap U_a$. We define
$H(\Gamma)=0$ by continuity.

\begin{lemma}\label{lem:H extended}
Extension of $H$ to $\Pi\cap U$ by flow of $X_0$ can be extended to a multivalued holomorphic
function defined in a neighborhood of $\Gamma$ in a complexification of $U$.
\end{lemma}

\begin{proof}
First, $H$ is analytic in some neighborhood of $\Gamma_1$, as it is an analytic
function extended by analytic flow of $X_0$. Choose local linearizing
coordinates $(x,y)$ near $a_2$ in such a way that $\Gamma_1=\{y=0\}$. By
assumption, $yx^{\mu}$ is the local first integral of $X_0$ near $a_2$.  Therefore
$H=f(yx^\mu )$, and, restricting to a transversal $x=x_0\ll 1$, one can see
that $f$ is analytic and invertible. Therefore $H$ can be extended to a
neighborhood of $a_2$.

Moreover, $\left(f^{-1}(H)\right)^{1/\mu}$ is an analytic local first integral
near the point  $y=1$ of $\Gamma_2$. Therefore it can be extended to a
neighborhood of $\Gamma_2$ (here we use that $U$ is orientable, so $\Gamma_2$
is different from $\Gamma_1$), and, as above, to a neighborhood of $a_3$ (here
we use that $a_2\not=a_3$), and so on.
\end{proof}

Note that from the above construction follows that near each $\Gamma_i$ the
first integral $H$ is equal, up to an invertible function,  to $x^{\lambda_i}$,
where $\{x=0\}$ is a local equation of $\gamma_i$. Also, near any singular
point of $\Gamma$ the first integral $H$ is equal, up to an invertible
function, to $x^\lambda y^\mu$.
\begin{corollary}\label{cor:dlogH is merom}
The one-form $\frac{dH}{H}$ is meromorphic one-form in $U$ with logarithmic
singularities only.
\end{corollary}

Assume that $n\ge 3$. One can easily construct a $C^\infty$ isomorphism of a
sufficiently small neighborhood   $U$ of $\Gamma$  with a  neighborhood of
a regular $n$-gone in $R^2$ in such a way that the image of $\Pi\cap U$ will lie
inside the $n$-gone and image of $\Gamma$ coincides with the $n$-gone.  Due to \cite{Grauert}, some neighborhood $U^{\C}$ of $U$
in its complexification is a Stein manifold. This implies that this isomorphism
can be chosen bianalytic. Similarly, for $n=2$ one can map bianalytically a
neighborhood of $U$ to a union of two arcs $\{x^2+(|y|-1)^2=2\}\subset\R^2$,
which, for the rest of the paper, will be called "regular 2-gone".

We transfer everything to plane using this isomorphism and will denote the
images on plane of the previously defined objects by the same letters. The
first integral $H$ takes the form $H=H_1\prod_{i=1}^n P_i^{\lambda_i}$, where
$P_i$ are analytic functions in $U$ with  $\{P_i=0\}=\Gamma_i$,  $H_1$ is an analytic functions non-vanishing in its
neighborhood $U$ and $\lambda_i>0$. Note that $H>0$ in the part of  $U$ lying
inside the $n$-gone. Further we assume that $H_1\equiv 1$, so $H=\prod P_i^{\lambda_i}$  (one can achieve this by e.g. taking
$P_1H_1^{1/\lambda_1}$ instead of $P_1$).

The family $X_\lambda$ becomes a family of planar analytic vector fields
defined in a neighborhood $U$ of a regular $n$-gone $\Gamma\subset\R^2$, and
$X_0$ has a first integral $H$ of Darboux type in $U$. Let $X_{\epsilon}=X_{\lambda(\epsilon)}$ be a one-parametric deformation of $X_0$ as in Proposition~\ref{ppm}. Define meromorphic forms $\omega^2$,
$\omega_{\epsilon}$ as
\begin{equation}
\omega^2(X_0,\cdot)=\frac{dH}{H},\qquad
\omega^2(X_{\epsilon},\cdot)=X_0+\omega_{\epsilon}.
\end{equation}
According to \cite[Theorem 2.1]{gav}, $M_k$ can be represented as a
linear combination of iterated integrals over $\{H=t\}$ of forms which are
combinations of Gauss-Manin derivatives of $\omega_\epsilon$.

Recall that the Gauss-Manin derivative of a form $\eta$ is defined as a form
$\eta'$ such that $d\eta=d(\log H)\wedge\eta' $. In general, $\eta'$ cannot be
uniquely defined from this equation, though its restrictions to $\{H=t\}$ are
defined unambiguously. However, since  $U^{\C}$ is Stein, in our situation one
can choose a meromorphic in $U$ representative of $\eta'$, with poles on
$\breve{\Gamma}$ only (where $\breve{\Gamma}$ is the union of lines containing
sides of $\Gamma$).

Therefore Theorem~\ref{main2} follows from the following claim
\begin{theorem}\label{main2'}
Let $H=\prod_{i=1}^n P_i^{\lambda_i}$ be as above, and let
$\gamma(t)\subset\{H=t\}$ be the connected component of its level set lying
inside $\Gamma$. Zeros of polynomials in iterated integrals
$I(t)=\int_{\gamma(t)}\omega_1...\omega_k$  corresponding to meromorphic
one-forms $\omega_1,...,\omega_k$ with poles in $\breve{\Gamma}$ cannot
accumulate to $0$.
\end{theorem}

From the above discussion it is clear that Theorem~\ref{main1} follows on its
turn from the following
\begin{theorem}\label{main1'}
Let
$$X_0= H_y
\frac{\partial}{\partial x} - H_x \frac{\partial}{\partial y}$$ where $H$ is a
real analytic function with isolated singularities in some complex neighborhood
of the closed period annulus $\bar{\Pi}= \{ \gamma(t): 0\leq t\leq 1\}$, where
$\gamma(t)\subset\{H=t\}$ is the connected component of the level set of $H$
lying inside $\Gamma$. Zeros of the first non-vanishing Poincar\'{e}-Pontryagin
function $M_k$, corresponding to a one-parameter analytic deformation
$X_\varepsilon$ of $X_0$ cannot accumulate to $0$.
\end{theorem}

\section{Non-oscillation  in the Hamiltonian case}
\label{proof1}
Here shall prove Theorem \ref{main1'}. This follows from the following two results
\begin{theorem}[\cite{gav}]
\label{th2}
The Poincar\'{e}-Pontryagin function $M_k$ satisfies a linear differential equation of a Fuchs
type in a suitable complex neighborhood of $0 \in \C$.
\end{theorem}
\begin{theorem} \label{th3}
The monodromy operator of the above Fuchs equation corresponding to a loop encircling the origin in $\C$ is
quasi-unipotent.
\end{theorem}
Let us recall that an endomorphism is called unipotent, if all its eigenvalues are equal to $1$, and
quasi-unipotent if  all of them are roots of the unity. The above theorems imply that the
Poincar\'{e}-Pontryagin-Melnikov function has a representation in a neighborhood of $u=0$
$$
M_k(u)= \sum_{i=0}^N\sum_{j=0}^N u^{\mu_i}(log (u))^j f_{ij}(u)
$$
where $N\in \N$, $\mu_j \in \Q$, and $f_{ij}$ are functions analytic in a neighborhood of $u=0$.  This shows
that the zeros of $M_k|_{(0,1)}$ do not accumulate to $0$. Of course, similar arguments hold in a neighborhood
of $u=1$, so $M_k$ has a finite number of zeros on $(0,1)$. This completes the proof of Theorem \ref{main2} in
the Hamiltonian case. To the end of the section we prove Theorem \ref{th3}. The open real surface $S$ is
analytic and hence possesses a canonical complexification. Similarly, any analytic family of analytic vector
fields $X_\lambda$ is extended to a complex family of vector fields, depending on a complex parameter. In this
section,  by abuse of notation, the base field will be $\C$. A real object and its complexification will be
denoted by the same letter.

Let $U\supset \bar{\Pi}$ be an open complex neighborhood of $ \bar{\Pi}$  in which the complexified vector field
$X_0$ has an analytic first integral $f$ with isolated critical points. The restriction of $f$ on the interval
$(0,1)$ (after identifying $\Pi$ to $S^1\times (0,1)$) is a local variable with finite limits at $0$ and $1$
Therefore we may suppose that $f(0)=0$, $f(1)=1$, and the restriction of $f$ to $(0,1)$ is the canonical local
variable on $(0,1)\subset \R$. The function $f$ defines a locally trivial Milnor fibration in a neighborhood of
every isolated critical point. There exists a complex neighborhood $U$ of $ \bar{\Pi}$ in which $F$ has only
isolated critical points. Moreover the compactness of $ \bar{\Pi}$ implies that there exists a complex
neighborhood $D\subset \C$ of the origin, homeomorphic to a disc, such that the fibration
\begin{equation}\label{milnor}
U\cap \{f^{-1}(D\setminus\{0\} ) \} \stackrel{f}{\rightarrow} D\setminus\{0\}
\end{equation}
is locally trivial, and the fibers $ f^{-1}(t)\cap U$ are open Riemann surfaces homotopy equivalent to a bouquet of a finite number of circles.
 Consider a one-parameter analytic deformation $X_\varepsilon$
of the vector field $X_0$. As $f$ is a first integral of $X_0$, then there exists an unique symplectic two-form
$\omega^2$, such that
$$
\omega^2(X_0,.) = df .
$$
Indeed, if in local coordinates
$$ X_0 = a\,\frac{\partial}{\partial x}+
b\,\frac{\partial}{\partial y}
$$
then $X_0.df=0$ implies $(a,b)=\lambda (f_y,-f_x)$, where $\lambda$ is analytic in $U$ and non-vanishing in
$\Pi$. It follows that
$$
\omega^2= \frac{dx \wedge dy}{\lambda} .
$$

Define a unique meromorphic one-form $\omega_\varepsilon$ by the formula
$$
\omega^2(X_\varepsilon,.) = df + \omega_\varepsilon .
$$
The one form $\omega_\varepsilon $ is meromorphic in $U$, depends analytically on $\varepsilon $, and
$\omega_0=0$. Its pole divisor does not depend on $\varepsilon $ as in the local variables above it is defined
by $\lambda=0$. Therefore $\omega_\varepsilon = \sum_{i\geq 1} \varepsilon^i \omega_i$ where $\omega_i$ are
given meromorphic one-forms in $U$ with a common pole divisor which does not intersect the period annulus $\Pi$.
In the complement of the singular locus of $X_\varepsilon$ the vector field $X_\varepsilon$ and the one form $df
+\omega_\varepsilon$ define the same foliation, and therefore define the same first return map associated to
$\Pi$. Denote this map by $P(t,\varepsilon)$, where $t\in(0,1)$ is the restriction of $f$ to a cross-section of
the period annulus $\Pi$ (this does not depend on the choice of the cross-section). We have
$$
P(t,\varepsilon)= t  + \sum_{k\geq 1} \varepsilon^k M_k(t) .
$$
On each leaf of the foliation defined by $X_\varepsilon$ we have $df =-\omega_\varepsilon$ which implies
$$
M_1(t) = \int_{\gamma_t} \omega_1
$$
where $\{\gamma_t\}_t$ is the family of periodic orbits (with appropriate orientation)  of $X_0$,
$\Pi=\cup_{t\in(0,1)} \gamma_t$, \cite{pontryagin}. Thus the first Poincar\'{e}-Pontryagin-Melnikov function is an
Abelian integral and its monodromy representation is straightforward. Namely, the meromorphic one-form
$\omega_1$ restricts to a meromorphic one-form on the fibers of the Minlor fibration (\ref{milnor}). We may also
suppose that $\omega_1|_{f^{-1}(t)}$ has a finite number of poles $\{P_i(t)\}_i$ (after choosing appropriately
the domain $U$). Denote
$$\Gamma_t= U\cap \{f^{-1}(t)\setminus \{P_i(t)\}_i\}$$

The Milnor fibration (\ref{milnor}) induces a representation
\begin{equation}\label{hmilnor}
\Z=\pi_1(D\setminus\{0\},*)\rightarrow Aut(H_1(\Gamma_t,\Z))
\end{equation}
which implies the monodromy representation of $M_1$. Suppose first that $\omega_1$ is analytic in $U$. It is
well known that the operator of the classical monodromy of an isolated critical point of an analytic function is
quasi-unipotent, e.g. \cite{katz}. Therefore the representation in $Aut(H_1(U\cap \{f^{-1}(t)\} ,\Z)) $ of a
small loop about $0$  in $\pi_1(D\setminus\{0\},*)$ is quasi-unipotent. More generally, let $\omega_1$
be meromorphic one-form with a finite number of poles on the fibers $U\cap \{f^{-1}(t)\}$. A monodromy operator
permutes the poles and hence an appropriate power of it leaves the poles fixed. Therefore this operator is
quasi-unipotent too and Theorem \ref{th2} is proved in the case $M_1\neq 0$. Of course, it is well known that an
Abelian integral has a finite number of zeros \cite{kho,var}.

Let $M_k$ be the first non-zero Poincar\'{e}-Pontryagin-Melnikov function. Its "universal" monodromy representation
was constructed in \cite{gav1}. For convenience of the reader we reproduce  it here. Recall first that $M_k(t)$
depends on the \textit{free} homotopy class of of the loop $\gamma_t$ in $\pi_1(\Gamma_t)$ \cite[Proposition
1]{gav1} and that this property does not hold true for the first return map $P(t,\varepsilon)$ (which depends on the homotopy class of $\gamma_t$ in $\pi_1(\Gamma_t,*)$) . Let
$F=\pi_1(\Gamma_t,*)$ be the fundamental group of $\Gamma_t$. It is a  finitely generated  free group. Let $\mathcal{O}\subset \pi_1(\Gamma_t)$ be the orbit of the loop $\gamma_t$ under the action of
$\Z^2=\pi_1(D\setminus\{0\},*)$ induced by (\ref{milnor}). The set $\mathcal{O}$ generates a normal subgroup
of $F$ which we denote by $G$. The commutator subgroup $(G,F)\subset F$ is the normal sub-group of $F$ generated
by commutators $(g,f)=g^{-1}f^{-1} g f$. The Milnor fibration (\ref{milnor}) induces a representation
\begin{equation}\label{fmilnor}
\Z=\pi_1(D\setminus\{0\},*)\rightarrow Aut(G/(G,F)) .
\end{equation}
According to \cite[Theorem 1]{gav1}, the monodromy representation of $M_k$ is a sub-representation of the
monodromy representation dual to (\ref{fmilnor}). Unfortunately the free Abelian group $F/(G,F)$ is not
necessarily of finite dimension. To obtain a finite-dimensional representation we use the fundamental fact that
$M_k$ has an integral representation as an iterated path integral of length $k$ \cite[Theorem 2.1]{gav}.

To use this, define by induction $F_{i+1}=(F_i,F)$, $F_1=F$. We will later consider the associated graded group
\begin{equation}\label{graded}
gr F = \bigoplus_{i=1}^\infty gr ^i F, \,\, gr ^i F=  F_i/F_{i+1} .
\end{equation}
 It is well know that an iterated integral
of length $k$ along a loop contained in $F_{k+1}$ vanishes identically. Therefore, to study the monodromy representation of $M_k$, we shall  truncate with respect to
$F_{k+1}$ and obtain a finite-dimensional representation. Namely, for every subgroup $H\subset F$ we denote
$$\tilde{H}= (H \cup F_{k+1}) / F_{k+1} .$$ The representation (\ref{fmilnor}) induces a homomorphism
\begin{equation}
\label{monodromy2} \pi_1(\C\setminus D , *) \rightarrow Aut(\tilde{G}/(\tilde{G},\tilde{F}))
\end{equation}
and the monodromy representation of $M_k$ is a sub-representation of the representation dual to
(\ref{monodromy2}) \cite{gav}. The Abelian group $\tilde{G}/(\tilde{G},\tilde{F})$ is, however, finitely
generated. Indeed the lower central series of $\tilde{F}= \tilde{F_1}$ is
$$ \tilde{F_1} \supseteq \tilde{F}_2  \supseteq \dots \tilde{F_k} \supseteq \{id\}
$$
and hence $\tilde{F}$ is a finitely generated nilpotent group. Each sub-group of such a group is finitely
generated too, e.g. \cite{hall}.

The central result of this section is the following theorem,  from which
 Theorem \ref{th3} follows immediately
\begin{proposition}
\label{th4} The monodromy representation
(\ref{monodromy2}) is quasi-unipotent.
\end{proposition}
Indeed, $M_k$ satisfies a Fuchsian equation on $D$, whose monodromy representation is a sub-representation of
the representation dual to (\ref{monodromy2}) \cite[Theorem 1.1]{gav} and \cite[Theorem 1]{gav1}. To prove
Proposition \ref{th4} we recall first some basic facts from the theory of free groups, e.g. Serre \cite{serre}, Hall
\cite{hall}. The graded group $gr F$ (\ref{graded}) associated to the free finitely generated group $F$
 is a Lie algebra with a bracket induced by the commutator $(.,.)$ on $F$. The Milnor fibration (\ref{milnor})
induces a representation
\begin{equation}\label{grmilnor}
\Z=\pi_1(D\setminus\{0\},*)\rightarrow  Aut_{Lie}(gr  F)
\end{equation}
where $Aut_{Lie}(gr  F)$ is the group of Lie algebra automorphisms of $gr  F$. Let $l$ be a generator of $
\pi_1(D\setminus\{0\},*)$. It induces   automorphisms  $l_*\in
Aut_{Lie}(gr  F)$ and $l_*|_{gr ^k F}\in Aut(gr ^k F)$. We note that $gr ^1 F = H_1(\Gamma_t,\Z)$ and hence
$l_*|_{gr ^1 F}$ is quasi-unipotent.
\begin{proposition}
\label{quasi}
Let $l_* \in Aut_{Lie}(gr  F)$ be such that $l_*|_{gr ^1 F}$ is quasi-unipotent. Then for every
$k\geq1$ the automorphism $l_*|_{gr ^k F}$ is quasi-unipotent.
\end{proposition}
The proof is by induction. Let $X=\{x_1,x_2,...,x_\mu\}$ be the free generators of $F$ and consider the free Lie
algebra $L_X$ on $X$. It is a Lie sub-algebra of the associative non-commutative algebra of polynomials in the
variables $x_i$ with a Lie bracket $[x,y]=xy-yx$. The canonical map $(x,y)\mapsto [x,y]$ induces an
 isomorphism of Lie algebras $gr  F \rightarrow L_X$, \cite[Theorem 6.1]{serre}. Let $L_X^k \subset L_X$ be the graded piece
 of degree $k$. We shall show that $l_*|_{L_X ^k }$ is quasi-unipotent. The proof is by induction. Suppose that
 the restriction of $l_*$ on $gr ^1 F = L_X^1=H_1(\Gamma_t,\Z)$ is quasi-unipotent, i.e. for some $p,q$, the restriction of
 $(l_*^p-id)^q$ on $gr ^1 F$ is $0$. The operator $Var_*= l_*^p-id$ is a linear automorphism, but not a Lie
 algebra automorphism. The identity

\begin{eqnarray*}
 Var_* [x,y] &=&(l_*^p-id)(xy-yx)= l_*^p \,x l_*^py- l_*^p \,y l_*^px -xy +yx \\
   &=& [Var_* x,Var_* y]+ [Var_* x,y]+
[x,Var_* y]
\end{eqnarray*}
shows that the restriction of $Var^{2q}$ on $L_X^2$ vanishes identically. Therefore The automorphism $l_*$
restricted to $L_X^2$ or $gr ^2 F$ is quasi-unipotent. The case $k\geq 3$ is similar. Proposition \ref{quasi} is
proved.$\Box$

According to the above Proposition for every $i\in \N$ there are integers $m_i,n_i$, such that the polynomial
$p_i(z)=(z^{m_i}-1)^{n_i}$ annihilates $l_*|_{gr ^k F}$. Proposition \ref{th4} will follow on its hand from the following
\begin{proposition}
\label{prop3}
The polynomial $ p= \prod_{i=1}^k p_i$ annihilates  $l_*\in Aut(\tilde{G}/(\tilde{G},\tilde{F}))$.
\end{proposition}
\textbf{Proof.}
Let $l\in \pi_1(D\setminus \{0\},*)$. It induces an automorphism of the Abelian groups  $G/(G,F ),  G \cap F_i / (  G \cap F_i ,F ), F_i/F_{i+1}$ denoted, by abuse of notation, by $l_*$. We denote by $p_i(l_*) = (l_*^{m_i}-id)^{n_i}$ the corresponding homomorphisms.
It follows from the definitions that the  diagram (\ref{diagram})   of Abelian groups, is commutative (the vertical arrows are induced by the canonical projections). Therefore if an equivalence classe
$[\gamma] \in G/(G,F)$ can be represented by a closed loop $\gamma \in F_i$, then $p_i(l_*)[\gamma]$ can be
represented by a closed loop in $F_{i+1}$. Therefore for every  $[\gamma] \in G/(G,F)$, the equivalence class
$p(l_*)$ can be represented by a closed loop in $F_{k+1}$. In other words $p(l_*)$ indices the zero automorphism
of $Aut(\tilde{G}/(\tilde{G},\tilde{F}))$.$\Box$

\begin{equation}\label{diagram}
\xymatrix{
   F_i/(F_i,F) \ar[r]^{p_i(l_*) }  & F_i/(F_i,F) \\
     G \cap F_i / (  G \cap F_i ,F ) \ar[r]^{p_i(l_*) }  \ar[u]_{\pi_2} \ar[d] _{\pi_1} &  G \cap F_i / (  G \cap F_i ,F)  \ar[u]_{\pi_2} \ar[d] _{\pi_1} \\
     G/(G,F )  \ar[r]^{p_i(l_*) } & G/(G,F )\\}
\end{equation}

\section{Non-oscillation  in the Darboux case}
\label{proof2}

In this section we prove Theorem~\ref{main2'}. First, we consider
\emph{elementary iterated integrals} - the iterated integrals over the piece of
the cycle lying near the saddles. We give a representation of the Mellin
transform of the elementary iterated integral as a converging multiple series.
This representation provides an asymptotic series for the elementary iterated
integral, with some explicit estimate  of the error, see
Theorem~\ref{thm:elemint} below.

The general iterated integral of length $k$ turns out to be a polynomial
(depending on $X_0$ and $k$ only) in elementary iterated integrals, by
Lemma~\ref{lem:mellin of iterated}. We give analogue of the estimates of
Theorem~\ref{thm:elemint} for such polynomials. This allows to prove a
quasianalyticity property: if the asymptotic series corresponding to the
iterated integral is zero, then the integral itself is zero. This implies
Theorem~\ref{main2'} since the zeros of the partial sums of the asymptotic
series do not accumulate to $0$, see Corollary~\ref{cor:quasian elemint}.

The arguments follow the pattern of \cite{novikov}, so the proofs are replaced
by a reference whenever possible.

\subsection{Iterated integral as a polynomial in elementary iterated integrals.}

Let $\gamma(u)$, $u\in[0,1]$, be a parameterization of the cycle
$\gamma_t\subset\{H=t\}$ (we fix some $t>0$ for a moment). As in
\cite{novikov}, the cycle of integration can be split into several pieces
$\gamma_j$, those lying near the sides of the polycycle, and those near the
vertices. We assume that the vector field can be linearized in the charts
containing these pieces, and call these pieces elementary. Let
$0=v_0<v_1<...<v_m<1$ be the parameterization of the ends of these pieces.

The iterated integral in the parameterized form is equal to
$$\int_{\Delta}g_1(u_1)...g_k(u_k) du_1...du_k,$$ where $\Delta=\{0\le
u_1\le...\le u_k\le1\}\subset\R^k$ is a simplex.

Consider connected components of the complement of $\Delta$ to the union of
hyperplanes $\cup_{i,j}\{u_j=v_i\}$. Each connected component can be defined
as
$$\{0\le u_1\le ...\le
u_{i_1}<v_1<u_{i_1+1}\le...<v_m<u_{i_m+1}\le...\le u_k\le1\},$$ i.e. is a
product $\Delta_1\times...\times \Delta_m$ of several simplices of smaller
dimension of the form $\Delta_j=\{v_j< u_{i_j+1}\le...\le u_{i_{j+1}}<
v_{j+1}\}$. Therefore, by Fubini theorem, integral of $g_1(u_1)...g_k(u_k)$
over this connected component is equal to the product of integrals
$\int_{\Delta_j}g_{i_j+1}...g_{i_{j+1}}du_{i_j+1}... du_{i_{j+1}}$, i.e. to the
product of iterated integrals $\int_{\gamma_j}\omega_{i_j+1}...\omega_{i_j+1}$.

Let us call the iterated integral over an elementary piece $\gamma_j$ an
\emph{elementary iterated integral}. The above arguments show that
\begin{lemma}\label{lem:mellin of iterated}
Iterated integral is a polynomial with integer coefficients in elementary
iterated integrals. The polynomial depends on the length of the iterated
integral only.
\end{lemma}
The above arguments give an explicit form of this polynomial (though we will
not need it).

\subsection{Mellin transform of elementary iterated integrals}
There are two types of  elementary pieces: those lying in charts covering sides
of the polycycle, and those lying in charts  covering saddles. Similarly to
\cite{novikov}, the elementary iterated integrals corresponding to the pieces
of the first type are just meromorphic functions of the parameter on the
transversal, i.e. of $t^{1/\lambda_i}$.

From this moment we assume that the elementary piece lies near the saddle
$\{P_1=P_2=0\}$. In other words, we assume that
$\gamma(t)=\{x^{\lambda_1}y^{\lambda_2}=t\}\cap\{0\le x,y \le 1\}$.

We give description of iterated integrals in terms of their Mellin transforms. Recall that the Mellin transform of a function $f(t)$ on the interval $[0,1]$ is defined as  $\M f(s)=\int_0^1t^{s-1}f(t)dt$.
To describe the Mellin transform of the elementary iterated integrals over
$\gamma(t)$ let us  introduce a generalized compensator. We denote in this
section by $l$ the length of the elementary iterated integral.
 For $l\in \N$ and $\alpha=(m_1,n_1,....,m_l,n_l)\in
\Z^{2l}$ we define $\ell^l_{\alpha}(s;\lambda_1,\lambda_2)$ as
\begin{equation}\label{eq:compensator}
\ell^l_{\alpha}(s;\lambda_1,\lambda_2)=\prod_{j=0}^l\left(s+\lambda_1^{-1}\sum_{i=1}^jm_i+
\lambda_2^{-1}\sum_{i=j+1}^l n_i\right)^{-1}.
\end{equation}
We call $\M^{-1}\ell^l_{\alpha}(s;\lambda_1,\lambda_2)$ a generalized
compensator. Particular case of $l=1$ corresponds to the Ecalle-Roussarie
compensator. Generalized compensator is  a finite linear combination of
monomials of type $t^\mu(\log t)^{l'}$, for $l'\le l$.

We omit $\lambda_1,\lambda_2$ from the notation till the end of the section.

\begin{lemma}\label{lem:mellin elemint}
After some rescaling of $t$ the Mellin transform of an elementary iterated
integral is given by the following formula:
\begin{equation}\label{eq:mellin elemint}
\M\int\omega_1...\omega_l=\sum_{\alpha}c_\alpha \ell^l_{\alpha}, \quad
\alpha\in\left(\Z_{>-M}\right)^{2l},
\end{equation}
where $M$ is an upper bound for the order of poles of  $\omega_i$. Moreover,
$|c_{\alpha}|\le C2^{-|\alpha|}$.

\end{lemma}
This is a straightforward generalization of the construction of \cite{novikov},
which corresponds to $l=1$.

\begin{proof}
In the linearizing coordinates the first integral is written as $H=
x^{\lambda_1}y^{\lambda_2}$.  The Mellin transform of the iterated integral can
be computed explicitly for monomial forms $\omega_i=x^{m_i-1}y^{n_i}\,dx$:
\begin{eqnarray}\label{eq:elem monomial int}
\M\int\omega_1...\omega_l=\phantom{\hskip7cm}\\
\int_0^1t^{s-1}\int_{t^{1/\lambda_1}}^1x_1^{m_1-1}y_1^{n_1}
\int_{x_1}^1x_2^{m_2-1}y_2^{n_2}\int_{x_2}^1...\int_{x_{l-1}}^1x_l^{m_l-1}y^{n_l}
dx_l...dx_1dt=\\
=\int_0^1 t^{\frac{n_1+...+n_l}{\lambda_2}}t^{s-1}
\int_{t^{1/\lambda_1}}^1x_1^{m_1-1-n_1 \mu}
\int_{x_1}^1...\int_{x_{l-1}}^1x_l^{m_l-1-n_l\mu}
dx_l...dx_1dt=\\
=\int_0^1x_l^{m_l-1-n_l\mu}\int_0^{x_l}x_{l-1}^{m_{l-1}-1-n_{l-1}\mu}...
\int_0^{x_1^{\lambda_1}}t^{\frac{n_1+...+n_l}{\lambda_2}+s-1}dt...dx_l=\\
=\lambda_1^{-l}\prod_{j=0}^l\left(s+\lambda_1^{-1}\sum_{i=1}^jm_i+
\lambda_2^{-1}\sum_{i=j+1}^l n_i\right)^{-1}=\lambda_1^{-l}\ell^l_{\alpha}.
\end{eqnarray}
Similar formula holds for $\omega_i=x^{m_i}y^{n_i-1}\,dy$.

After rescaling of $H$ we can assume that the linearizing chart covers the
bidisk $\{0\le |x|,|y|\le 2\}$. Then the coefficients of the forms $\omega_i$
are meromorphic  in the bidisk, with poles on $\{xy=0\}$ of order at most $M$.
So $\omega_i$ can be represented as a convergent power series
$$
\omega_i=\sum_{m_i,n_i\in\Z_{>-M}}\left(c'_{i,m_i,n_i}x^{m_i-1}y^{n_i}\,dx+
c''_{i,m_i,n_i}x^{m_i}y^{n_i-1}\,dy\right),
$$
with coefficients $c'_{i,m_i,n_i},c''_{i,m_i,n_i}$ decreasing as
$O(2^{-m_i-n_i})$. Therefore the elementary iterated integral is a converging
sum of elementary iterated integrals of monomial forms, with coefficients being
products of $c'_{i,m_i,n_i},c''_{i,m_i,n_i}$, $i=1,...,l$ and $m_i,n_i\in \N$.
From (\ref{eq:elem monomial int}) one gets upper bounds for the elementary
iterated integrals of monomial forms, which guarantees that one can perform
Mellin transform termwise, and we get the required formula.\end{proof}

As in \cite{novikov}, one can check that the inverse Mellin transform of Mellin
transforms of elementary iterated integrals can be defined as
\begin{equation}\label{eq:invmellin}
\M^{-1}g=\frac 1 {2\pi i}\int_{\partial \Pi}t^{-s}g(s) ds, \quad \Pi=\{\Re s\le
M<+\infty, |\Im s|\le 1\},
\end{equation}
where $M$ is sufficiently big. Indeed, $|\ell^l_{\alpha}(s)|\le 1$ on $\Pi$, so
(\ref{eq:mellin elemint}) converges uniformly on this contour, so one can
integrate the series (\ref{eq:elem monomial int}) termwise. However, for each
term (\ref{eq:invmellin}) does define the inverse Mellin transform, as each term is just a rational function in $s$.
\begin{corollary}\label{cor:elemint sum}
An elementary iterated integral can be represented as a convergent sum
\begin{equation}\label{eq:elemint sum}
\int\omega_1...\omega_l=\sum_{\alpha}c_\alpha \M^{-1}\ell^l_{\alpha}.
\end{equation}
\end{corollary}

The following estimate is the keystone of the proof, since it allows to
estimate the difference between the elementary iterated integral and the
partial sum of its asymptotic series.

\begin{lemma}\label{lem:upper bound elemint}
Let $I=\in\omega_1...\omega_l$ be an elementary iterated integral, and let $C$
be defined as in \ref{lem:mellin elemint}.  For any $s\in\C$ denote by
$\rho(s)$ the minimal distance from $S$ to the poles of $\M I$.

Then $|\M I(s)|\le C\rho(s)^{-l}$.
\end{lemma}
\begin{proof} Indeed, the absolute value of each term in the sum in
(\ref{eq:mellin elemint}) can be estimated from above as
$|c_{\alpha}|\rho(s)^{-l}$, and the estimate follows from $|c_{\alpha}|<C2^{-\alpha}$.\end{proof}

\subsection{Asymptotic series of  elementary iterated integrals}

Inverse Mellin transform of $\ell^l_{\alpha}$ is a linear combination of
monomials of the type $t^\mu(\log t)^j$, where $\mu\in
\lambda_1^{-1}\Z+\lambda_2^{-1}\Z$ and  $0\le j\le l$. Collecting similar terms
in the expression for the elementary iterated integral $I$ together, we get a
formal series $\hat{I}$ of such terms, possibly divergent:
\begin{equation}\label{eq:formal elemint}
\hat{I}=\sum_{\mu, j} \hat{c}_{\mu, j}t^\mu(\log t)^j, \quad\text{where}\quad
\mu\in \lambda_1\Z_{>-M}+\lambda_2\Z_{>-M},\quad 0\le j\le l.
\end{equation}

\begin{theorem}\label{thm:elemint}
$\hat{I}$ is an asymptotic series of $I$. Moreover, for each $p\in\N$ there
exists $s_p\in[p,p+1]$ such that the partial sums $\hat{I}_p=\sum_{j,\mu<s_p}
\hat{c}_{\mu, j}t^\mu(\log t)^j$ of $\hat{I}$ satisfy the following:
\begin{equation}\label{eq:tail elemint}
|I(t)-\hat{I}_p(t)|\le C s_p^{l^2}t^{s_p},\quad t\in[0,1]
\end{equation}
where $C$ depends on $I$ but not on $p$.
\end{theorem}
The proof is a generalization of the proof of the corresponding statement from
\cite{novikov}.
\begin{proof}
Poles of $\M I$ are of the form $-\lambda_1^{-1}\sum_{i=1}^jm_i-
\lambda_2^{-1}\sum_{i=j+1}^l n_i$. Since $\lambda_1,\lambda_2>0$, there are
$O(p^{l-1})$ poles on the interval $J_p=[-p-1,-p]$, $p\in\N$. Therefore on each
interval $J_p$ one can find a point $-s_p$ such that
$\rho(-s_p)>O(p^{1-l})=O(s_p^{1-l})$.

For each $p\in \N$ let us split the contour of integration ${\partial \Pi}$
into two parts: boundary of $\Pi_p'=\{-s_p\le \Re s\le M, |\Im s|\le 1\}$ and
boundary of $\Pi_p=\{\Re s\le -s_p, |\Im s|\le 1\}$. Computing residues, we see
that $\frac 1 {2\pi i}\int_{\partial\Pi_p'}t^{-s}\M I ds$ is a partial sum
$\hat{I}_p(t)$ of $\hat{I}$ as defined above. Therefore
$I(t)-\hat{I}_p(t)=\frac 1 {2\pi i}\int_{\partial\Pi_p}t^{-s}\M I ds$. By
Lemma~\ref{lem:upper bound elemint}, $|\M I(s)|\le O(p^{l^2})$ on
$\partial\Pi_p$, and (\ref{eq:tail elemint}) follows.
\end{proof}

\subsection{Iterated integrals}
Here we extend the Theorem~\ref{thm:elemint} to the algebra $\A$ generated by
elementary iterated integrals.

Let $f=P(I_1,...,I_k)\in\A$ be an element in $\A$, where $P\in\C[u_1,...,u_k]$
and $I_1,...,I_k$ are elementary integrals. Substitution of convergent series
from (\ref{eq:elemint sum}) instead of $I_1,...,I_k$ gives a representation of
$f$ as a converging multiple sum of products (of length at most $k$) of
generalized compensators. Collecting similar terms, we obtain a formal series
$\hat{f}$ similar to (\ref{eq:formal elemint}), probably divergent.

\begin{theorem}\label{thm:int}
For any $p\in\N$ there exists $s_p\in[p,p+1]$ such that the partial sum
$\hat{f}_p$ of $\hat{f}$ satisfies the following
\begin{equation}\label{eq:tail int}
|f-\hat{f}_p|\le C s_p^dt^{s_p}
\end{equation}
for some $C,d$ independent of $p$.
\end{theorem}

Before proof of Theorem~\ref{thm:int} let us show that it implies
Theorem~\ref{main2'}.

\begin{corollary}\label{cor:quasian elemint}
 Let $f\in\A$. If $\hat{f}=0$, then $f\equiv 0$ on $[0,1]$. Also, isolated zeros of $f$ cannot accumulate to $0$.
\end{corollary}
\begin{proof} To prove the first claim, take a limit as $s_p\to+\infty$ in (\ref{eq:tail int}.

Now, if $f\not\equiv0$, then for some $\mu$ we have $|f-t^\mu P(\log
t)|=o(t^\mu)$ with some non-zero polynomial $P$ (where $-\mu$ is the rightmost
pole of $\M f$). This clearly implies the second claim. \end{proof}

The proof of Theorem~\ref{thm:int} occupies the rest of the paper.

\subsubsection{Mellin transform of a product of several generalized
compensators}

For $V=(v_1,...,v_n)\in\R^n$ define $\ell_v(s)=\prod_{i=1}^n(s+v_i)^{-1}$. Let
$V^j=(v^j_1,..,v^j_{n_j})\in\R^{n_j}$, $j=1,..,k$ and define
$\Phi(V^1,...,V^k)(s)= \M\left[\prod\left(\M^{-1}\ell_{V^j}\right)\right]$.

This is a rational function of $s$. We want to show that it depends
polynomially on $\{V^j\}$. Let $\Kappa$ denotes
the set of  function $\kappa:\{1,...,k\}\to\Z$ with the condition
$\kappa(j)\in\{1,...,n_j\}$, and define
$w_{\kappa}=v^1_{\kappa(1)}+...+v^k_{\kappa(k)}$.

\begin{lemma}\label{lem:univpol}
Let $S=S(V^1,...,V^k)=\prod_{\kappa\in\Kappa}(s+w_\kappa)$ be a polynomial in
$\R[V^1,...,V^k;s]$. There exists a polynomial
$R=R_{n_1,...,n_k}\in\R[V^1,...,V^k;s]$ such that $\Phi(V^1,...,V^k)(s)=R
S^{-1}$, $deg_sR<\deg_sS$.
\end{lemma}
\begin{proof}
By continuity of both sides it is enough to prove this for a dense subset of
$\prod \R^{n_j}$ consisting of non-resonant tuples $(V^1,...,V^k)$, namely for those  those
tuples for which all  $w_\kappa$ are different.

Let $\C_{\infty}(s)$ be the  ring of rational functions in $s$ vanishing at
infinity, and define convolution $f_1\ast f_2$ for  $f_1,f_2\in\C_{\infty}(s)$
by extending the rule
$$
\frac 1 {s+a}\ast\frac 1 {s+b}=\frac 1 {s+a+b}
$$
by linearity and continuity to the whole $\C_{\infty}(s)$ (in particular,
$(s+a)^{-k}\ast(s+b)^{-l}=(s+a+b)^{-k-l+1}$). Thus defined convolution is
Mellin-dual to the usual product. Therefore
$\Phi(V^1,...,V^k)(s)=\ell_{V^1}\ast...\ast\ell_{V^k}$. Decomposing each factor
into simple fractions
$$
\ell_{V^j}=\sum_i\frac {Res_{v^j_i}\ell_V^j}{s+v^j_i},\quad
Res_{v^j_i}\ell_V^j=\left(\prod_{i'\not=i}(v^j_i-v^j_{i'})\right)^{-1}
$$
and opening brackets, we see that
$$
 \Phi(V^1,...,V^k)(s)=\sum_{\kappa\in\Kappa}\frac{\prod_{j=1}^k Res_{v^j_{\kappa(j)}}\ell_V^j}{s+w_\kappa}.
$$
Reducing to a common denominator, we see that $ \Phi(V^1,...,V^k)(s)$ is a
rational function in $v^j_i,s$, with denominator dividing
$S\prod_{i,i',j}(v^j_i-v^{j}_{i'})$.

We claim that the factors $(v^j_i-v^{j}_{i'})$ do not enter denominator of $
\Phi(V^1,...,V^k)(s)$. Indeed, presence of such factor would mean that $ \Phi(V^1,...,V^k)(s)$ becomes unbounded as $v^j_i$ tends to
$v^{j}_{i'}$ for each
$s\in\C$ , which is not true:  for any tuple $(V^1,...,V^k)$ and every
sufficiently big  $s\in \R$ the function $ \Phi(V^1,...,V^k)(s)$ is locally
bounded near $(V^1,...,V^k,s)$.

\end{proof}

\subsection{Mellin transform of a product of elementary iterated integrals}
Let $I=I_1...I_k$ be a product of several elementary iterated integrals, and
let  order of $I_j$ be $l_j$. Then using representation (\ref{eq:elemint sum})
for $I_j$ and opening brackets, we see that
\begin{equation}\label{eq:monomial sum}
\M I=\sum_{\alpha_1,...,\alpha_k}c_{\alpha_1}...c_{\alpha_k}
\M\left(\prod_{j=1}^k\M^{-1}(\ell^{l_j}_{\alpha_j})\right),
\end{equation}
where $\alpha_j\in\left(\Z_{>-M}\right)^{2l_j}$.

\begin{lemma}\label{lem:upper bound monomial}
Let $\rho(s)$ be the distance from $s$ to the set of poles of $\M I$. Then $|\M
I(s)|\le C \rho^{-\prod l_j}(|s|+1)^d$ for some $d>0$.
\end{lemma}
\begin{proof}
Let us estimate from above the terms
$\M\left(\prod_{j=1}^k\M^{-1}(\ell^{l_j}_{\alpha_j}\right)$ from
(\ref{eq:monomial sum}). By Lemma~\ref{lem:univpol} it is equal to
$R(V^1,...,V^k;s)/S(V^1,...,V^k;s)$, where $V^j=(v^j_1,...,v^j_{l_j})$ is
defined by
$$
v^j_i=-\lambda_{j1}^{-1}\sum_{p=1}^i m^j_p-\lambda_{j2}^{-1}\sum_{p=i+1}^{l_j}
n^j_p,\quad \alpha_j=(m^j_1,...,n^j_{l_j})\in \Z_{>-M}^{l_j},
$$
as in (\ref{eq:compensator}). This means that  $V^j=L_j\alpha_j$ for some
linear map $L_j:\R^{l_j}\to \R^{l_j}$. Therefore $R$ is a polynomial in
$(s;\alpha_1,...,\alpha_k)$, and
$$
|R(s)|\le \const(1+s)^d(1+\sum|\alpha_j|)^d,\quad\text{for  } d=\deg R\ge0.$$

From the other side, $S$ is a monic polynomial in $s$ of degree $\prod l_j$
with roots in the poles of $\M I$, so $|S(s)|\ge \left(\rho(s)\right)^{\prod
l_j}$. Taken together, this means that
\begin{equation}\label{eq:upper for term}
\left|\M\left(\prod_{j=1}^k\M^{-1}(\ell^{l_j}_{\alpha_j})\right)\right|\le
\const\left(\rho(s)\right)^{-\prod l_j} (1+s)^d(1+\sum|\alpha_j|)^d.
\end{equation}

Now, we know that $|c_{\alpha_j}|\le C2^{-|\alpha|}$ by Lemma~\ref{lem:mellin
elemint}, so we estimate $|\M I(s)|$ from above as
\begin{equation}\label{eq:upper for monomial}
|\M I(s)|\le \const\left(\rho(s)\right)^{-\prod l_j} (1+s)^d
\sum_{\alpha_1,...,\alpha_k}2^{-\sum|\alpha_j|}(1+\sum|\alpha_j|)^d,\end{equation}
which, by convergence of the series, proves the Lemma.
\end{proof}

\subsubsection{Proof of Theorem~\ref{thm:int}}
Let $I$ now be a polynomial in several elementary iterated integrals,
$I=P(I_1,...,I_k)$. The set of poles of the Mellin transform $\M I$ of $I$ is
the union of sets of poles of Mellin transforms of each monomial of $P$, so the
number of poles of $\M I$ on an interval $J_p=[-p-1,-p]$ counted with
multiplicities grows as some power of $p$.

This means that for each $p\in \N$ one can find $s_p\in J_p$ such that the
distance $\rho(s_p)$ from $p$ to the set of poles of $\M I$ will be bigger than
$|s_p|^{-d'}$ for some $d'>0$. Then splitting the contour of integration of the
inverse Mellin transform as in Theorem~\ref{thm:elemint}, we conclude from
Lemma~\ref{lem:upper bound monomial} that  $|\M I|<C|s_p|^{d''}$ on the
$\partial \Pi_p$ for some fixed $d''>0$, and the claim follows.


\end{document}